\documentclass{amsart}

\usepackage{amssymb,amsmath,amsthm}
\usepackage{graphics}
\usepackage[arrow, matrix, curve]{xy}
\usepackage{txfonts}
\usepackage{bbm}

\theoremstyle{definition}
\newtheorem{definition}{Definition}

\theoremstyle{plain}
\newtheorem{theorem}[definition]{Theorem}

\newtheorem{lemma}[definition]{Lemma}

\theoremstyle{remark}

\theoremstyle{dotless}

\newcommand\I{\mathbbm{1}}

\newcommand{\N}{\mathbb{N}}

\newcommand{\R}{\mathbb{R}}
\newcommand{\T}{\mathbb{T}}
\newcommand{\Z}{\mathbb{Z}}

\renewcommand{\P}{\mathcal{P}}

\newcommand{\eps}{\varepsilon}

\newcommand{\me}{^{-1}}

\title{An ultrafilter approach to Jin's Theorem}
 \author{Mathias Beiglb\"ock}
 \address{Fakult\"at f\"ur Mathematik, Universit\"at Wien\endgraf
Nordbergstra\ss e 15\\ 1090 Wien, Austria}
\email{mathias.beiglboeck@univie.ac.at}
\thanks{The author acknowledges financial support from the Austrian Science
Fund (FWF) under grant P21209.}

\begin{document}
\maketitle
\dedicatory{ This paper is dedicated to Dona Strauss on the occasion of her 75th birthday. }
\begin{abstract}
It is well known and not difficult to prove that if $C\subseteq\Z$ has positive upper Banach density, the set of differences $C-C$ is syndetic, i.e.\ the length of gaps is uniformly bounded. More surprisingly, Renling Jin showed that whenever $A$ and $B$ have positive upper Banach density, then $A-B$ is \emph{piecewise syndetic}. 

Jin's result follows trivially from the first statement provided that $B$ has large intersection with a shifted copy $A-n$ of $A$. Of course this will not happen in general if we consider shifts by integers, but the idea can be put to work if we allow ``shifts by ultrafilters''. As a consequence we obtain Jin's Theorem.
\end{abstract}

The \emph{upper Banach density} of  $C\subseteq \Z$ is given by $d^*(C):= \overline\lim_{m-n\to \infty} \frac1{m-n+1}|C\cap \{n,\ldots, m\}|.$ A set $S\subseteq \Z$ is syndetic if the gaps of $S$ are of uniformly bounded length, i.e.\ if there exists some $k>0$ such that $S-\{-k, \ldots, k\}=\Z$. A set $P\subseteq \Z$ is piecewise syndetic if it is syndetic on large pieces, i.e.\ if there exists some $k\geq 0$ such that $P-\{-k, \ldots, k\}$ contains arbitrarily long intervals of integers. 

It was first noted by F\o lner (\cite{Foln54a,Foln54b}) that  $C-C=\{c_1-c_2:c_1, c_2\in C\}$ is syndetic provided that $d^*(C)>0$. (To see this, pick a subset $\{i_1, \ldots, i_m\}\subseteq \Z$ which is \emph{maximal} subject to the condition that  $C-i_1,\ldots, C-i_m$ are mutually disjoint. This is possible  since disjointness implies $d^*(C-i_1\cup\ldots\cup C-i_m)=m \cdot d^*(C)$, therefore $m$ is at most $1/d^*(C)$. But then maximality implies that for each  $n\in  \Z$ there is some $i_k$, $k\in \{1,\ldots, m\}$ such that $(C-n)\cap(C-i_k)\neq \emptyset$, resp.\ $n\in (C-C)+i_k$. Thus $\bigcup_{k=1}^m (C-C) +i_k=\Z$.)



 Simple counterexamples yield that the analogous statement fails when two different sets are considered, but Renling Jin discovered the following interesting result.
\begin{theorem}[\cite{Jin02}]\label{IntegerJin}
Let $A,B\subseteq \Z, d^*(A), d^*(B)>0$. Then $A-B$ is piecewise syndetic.
\end{theorem}
 In Section 1 we reprove Jin's Theorem by reducing it to the $C-C$ case. Subsequently we discuss some modifications of our argument which allow to recover the refinements resp.\ generalizations of Jin's result found in \cite{JiKe03, BeFW06, Jin08, BeBF09}. 
\section*{1\ \ \  Ultrafilter proof of Jin's Theorem}
 As indicated in the abstract, we aim to show that one can shift a given large set $A\subseteq \Z$ by an ultrafilter so that it will have large intersection with another, previously specified set.  We motivate the definition of this ultrafilter-shift by means of analogy: For $n\in\Z$  denote by $e(n)$ the principle ultrafilter on $\Z$ which corresponds to $n$ and notice that
$$A-n=\{k\in \Z : n\in A-k \}= \{k\in \Z :  A-k \in e(n) \}.$$
Given an ultrafilter $p$ on $\Z$, we thus define $\{k\in \Z : A-k \in p\}$ as the official meaning of ``$A-p$''.  
\begin{lemma}\label{ShiftT}
Let $A,B\subseteq \Z$. Then there exists an ultrafilter $p$ on $ \Z$ such that 
$$d^*\Big((A-p) \cap B \Big)\ =\ d^*\Big(\{k:A-k\in p\}\cap B\Big)\ \geq\  d^*(A)\cdot d^*(B).$$
\end{lemma}
The proof of Lemma \ref{ShiftT} requires some preliminaries.

For a fixed set $A\subseteq \Z$ we can choose an invariant mean, i.e.\ a shift invariant finitely additive probability measure $\mu$ on $(\Z, \P(\Z))$) such that $d^*(A)=\mu(A)$. 
This is well known among aficionados (see for instance \cite[Theorem 5.8]{Berg06}) and not difficult to prove. First pick a sequence of finite intervals $I_n\subseteq \Z$ such that $|I_n|\to \infty$ and $d^*(A)= \lim_{n\to \infty} \mu_n(A)$ where $\mu_n(B):=
\frac1{|I_n|}|B\cap I_n|$ for $B\in \P(B)$.  Then let $\mu$ be a cluster point of the set $\{\mu_n: n\in\N\} $  in the (compact) product topology of $[0,1]^{\P(\Z)}$.




We consider the Stone-\v Cech compactification $\beta \Z$ of  the discrete space $\Z$. For our purpose it is convenient to view $\beta \Z$ as the set of all ultrafilters on $\Z$. By identifying integers with principal ultrafilters, $\Z$ is naturally embedded in $\beta \Z$. A clopen basis for the topology is  given by the sets $\overline C:= \{p\in\beta \Z : C\in p\}$, where $C\subseteq \Z$. 
A mean $\mu$ on $\Z$ gives rise to a positive linear functional $\Lambda$ on the space $B(\Z)$ of bounded functions on $\Z$. Making the identification $B(\Z)\cong C(\beta\Z)$ we find that, by the Riesz representation Theorem, there exists a regular Borel probability measure $\tilde\mu$ on $\beta \Z$ which corresponds to the mean $\mu$ in the sense that $\mu(A)=\tilde \mu\big(\overline A\big)$ for all $A\subseteq \Z$. (This procedure is carried out in detail for instance in \cite[p 11]{Pate88}.)



\begin{proof}[Proof of Lemma \ref{ShiftT}.]
Pick a sequence of intervals $I_n\subseteq \Z,|I_n|\uparrow \infty $ such that  $d^*(B)=\lim_{n}Ê\frac {|I_n\cap B|}{|I_n|}$. Pick an invariant mean $\mu$ such that $\mu (A)=d^*(A)$. Define $f_n:\beta \Z\to [0,1]$ by
$$f_n(p):=\frac1{|I_n|} \sum_{k\in  I_n\cap B} \I_{\overline {A-k}}(p)=\frac {|I_n\cap B \cap \{k:A-k\in p\}|}{|I_n|} $$ and set $f(p):= \overline \lim_n f_n(p)\leq d^*(B\cap \{k:A-k\in p\}).$  By Fatou's Lemma
$$ \int f\, d\tilde \mu\geq \overline{\lim_{n\to \infty}} \int \frac1{|I_n|} \sum_{k\in I_n\cap B}  \I_{\overline {A-k}}\, d\tilde \mu=  
\overline{\lim_{n\to \infty}} \frac1{|I_n|} \sum_{k\in I_n\cap B}  \mu( {A-k})
=  d^*(A) \cdot d^*(B),$$
thus there exists $p\in \beta \Z$ such that $d^*(A)\cdot d^*(B)\leq f(p).$
\end{proof}
The above  application of Fatou's Lemma is inspired by the proof of \cite[Theorem 1.1]{Berg85}.


\begin{proof}[Proof of Theorem \ref{IntegerJin}.]
Assume that $d^*(A), d^*(B)>0$. According to Lemma \ref{ShiftT}, pick an ultrafilter $p$ such that $C:=(A-p)\cap B$ has positive upper Banach density. Then $S:=(A-p)-B\supseteq C-C$ is syndetic. Also $s\in (A-p)-B \ \Longrightarrow \ A-B-s\in p$. 
Thus for each finite set $\{s_1, \ldots, s_n\}\subseteq (A-p)-B$, we have $\bigcap_{i=1}^n A-B- s_i\in p$. In particular this intersection is non-empty, hence there exists $t\in \Z$ such that $t+\{s_1, \ldots, s_n\}\subseteq A-B$. 
Summing up, we find that $A-B$ is piecewise syndetic since it contains shifted copies of all finite subsets of the syndetic set $(A-p)-B$.
\end{proof}


\section*{2\ \ \ Jin's Theorem in countable amenable (semi-) groups}
Following Jin's original work, it was shown in \cite{JiKe03} that Theorem \ref{IntegerJin} is valid in a certain class of abelian groups (including in particular $\Z^d$) and in \cite{Jin08} that it holds in $\oplus_{i=1}^\infty \Z$. Answering a question posed in \cite{JiKe03} it is proved in \cite[Theorem 2]{BeBF09} that Jin's theorem extends to all countable groups in which the notion of upper Banach density can naturally be formulated, that is, to all countable amenable groups.  

There is overwhelming evidence (see in particular \cite{HiSt98}) that whenever ultrafilters can be used to prove a certain combinatorial statement, then this ultrafilter proof will automatically work in a quite abstract setup. In this spirit the approach of Section 1 effortlessly yields the just mentioned strengthenings of Jin's Theorem; in fact  it is not even necessary to restrict the setting to groups.

If a semigroup $(S,\cdot)$ admits left- and right F\o lner sequences\footnote{A sequence $(F_n)_{n\in\N}$ of finite sets in a semigroup $S$ is a left/right F\o lner sequence if $\lim_{n\to \infty} |(sF_n)\Delta F_n|/|F_n|=0$ resp.\ $\lim_{n\to \infty} |(F_ns)\Delta F_n|/|F_n|=0$ for all $s\in S$. The existence of F\o lner sequences is sometimes used to define if a countable group / semigroup is amenable. We note that all  abelian semigroups and all solvable groups fall in this class. See for instance \cite{Pate88}.}, the notions $d^*_L,d^*_R$  of left- resp.\ right upper Banach density are defined analogously to upper Banach density in $\Z$, but with left- resp.\ right F\o lner sequences taking the role of intervals $\{n,\ldots, m\}$. Arguing precisely as in Section 1 we then get that for all $A,B\subseteq S$ there is an ultrafilter $p$ on $S$ such that 
$d_L^*\Big(B\cap (q\me A)\Big)\ =\ d_L^*\Big(B\cap \{s:As\me\in p\}\Big)\ \geq\  d_R^*(A)\cdot d_L^*(B).$ Consequently we have:
\begin{theorem}\label{SemigroupJin}
Let $(S,\cdot)$ be a semigroup which admits left- and right F\o lner sequences and let $A,B\subseteq S$, $d_R^*(A), d_L^*(B)>0$. Then $AB\me=\{s\in S:\exists b\in B\ sb\in A\}$ is (right) piecewise syndetic.
\end{theorem}\nopagebreak[4]
A subset $P$ of a semigroup $(S,\cdot)$ is \emph{(right) piecewise syndetic }Êif there exists a finite set $K\subseteq S$ such that for each finite set $F\subseteq S$ there is some $t\in S$ such that $tF\subseteq PK\me$.

If $(S,\cdot) $ is an amenable group, Theorem \ref{SemigroupJin}  is equivalent to \cite[Theorem 2]{BeBF09} which asserts that in this setup $AB$ is (right) piecewise syndetic provided that $d^*_R(A),d^*_R(B)>0$.
\section*{3\ \ \ Connections with Bohr sets}
So far, our results about the structure of $A-B$ originated from the fact that for $d^*(C)>0$ the set $C-C$ is syndetic. F\o lner (\cite{Foln54a,Foln54b}) proved a much stronger assertion, namely that $C-C$ is ``almost'' a \emph{Bohr$_0$ set. } Bohr$_0$ sets are the neighborhoods of  $0$ in the Bohr topology. (Equivalently,  $U\subseteq \Z$ is a {Bohr$_0$ set} iff  there exist $\alpha_1, \ldots, \alpha_n\in \T= \R/\Z$ and $\eps>0$  such that $\{k\in \Z: \|k\alpha_1\|, \ldots,\|k\alpha_n\|< \eps\}\subseteq U$.) A \emph{Bohr set} is a translate of a Bohr$_0$ set. Every Bohr set is syndetic, but the converse fails badly. F\o lner showed that if $d^*(C)>0$, then  there exist a Bohr$_0$ set $U$ and set $N\subseteq \Z$ with $d^*(N)=0$ such that $C-C\supseteq  U\setminus N. $ 

 In analogy to piecewise syndetic sets Bergelson, Furstenberg and Weiss introduced piecewise Bohr sets. A set $P$ is piecewise Bohr if it is Bohr on arbitrarily large intervals, that is, if there exist a Bohr set $U$ and a  sequence of intervals $I_n\subseteq \Z, |I_n|\uparrow \infty$ such that $P\supseteq U\cap \bigcup_{n=1}^\infty I_n$. 
  F\o lner's Theorem, $d^*(C)>0$ trivially  implies that $C-C$ is piecewise Bohr. We reprove the following refinement of Jin's Theorem  obtained in \cite{BeFW06}. 
\begin{theorem}\label{BohrJin}
Let $A,B\subseteq \Z, d^*(A), d^*(B)>0$. Then $A-B$ is piecewise Bohr. 
\end{theorem}
 \begin{proof} As in the case of piecewise syndetic sets one readily shows that if a set $P$ contains a translate of every finite subset of a piecewise Bohr set, then $P$ is piecewise Bohr itself.  In Section 1 we have seen that if $d^*(A), d^*(B)>0$, then $A-B$ contains shifts of all finite pieces of a set $C-C, d^*(C)>0$. By F\o lner's Theorem the latter set is piecewise Bohr, hence $A-B$ is piecewise Bohr as well. \end{proof}

In the spirit of Section 2 it is possible to  proceed along these lines in an abstract countable amenable group. 
This then leads to the amenable analog of Theorem \ref{BohrJin} derived in \cite[Theorem 3]{BeBF09}. 

It is natural to ask whether one can make assertions about the combinatorial richness of the set $A-B$ beyond the fact that is piecewise Bohr. In a certain sense this is not possible. For  every piecewise Bohr set $P\subseteq \Z$, there exist sets $A,B\subseteq \Z$ with $d^*(A), d^*(B)>0$ such that $A-B\subseteq P$ (\cite[Theorem 4]{BeBF09}).

\def\ocirc#1{\ifmmode\setbox0=\hbox{$#1$}\dimen0=\ht0 \advance\dimen0
  by1pt\rlap{\hbox to\wd0{\hss\raise\dimen0
  \hbox{\hskip.2em$\scriptscriptstyle\circ$}\hss}}#1\else {\accent"17 #1}\fi}


\begin{thebibliography}{BFW06}

\bibitem[BBF]{BeBF09}
Mathias Beiglb{\"o}ck, Vitaly Bergelson, and Alexander Fish.
\newblock Sumset phenomenon in countable amenable groups.
\newblock {\em Adv. Math., to appear}.

\bibitem[Ber85]{Berg85}
Vitaly Bergelson.
\newblock Sets of recurrence of {${\bf Z}\sp m$}-actions and properties of sets
  of differences in {${\bf Z}\sp m$}.
\newblock {\em J. London Math. Soc. (2)}, 31(2):295--304, 1985.

\bibitem[Ber06]{Berg06}
Vitaly Bergelson.
\newblock Combinatorial and {D}iophantine applications of ergodic theory.
\newblock In {\em Handbook of dynamical systems. {V}ol. 1{B}}, pages 745--869.
  Elsevier B. V., Amsterdam, 2006.
\newblock Appendix A by A. Leibman and Appendix B by Anthony Quas and
  M{\'a}t{\'e} Wierdl.

\bibitem[BFW06]{BeFW06}
Vitaly Bergelson, Hillel Furstenberg, and Benjamin Weiss.
\newblock Piecewise-{B}ohr sets of integers and combinatorial number theory.
\newblock In {\em Topics in discrete mathematics}, volume~26 of {\em Algorithms
  Combin.}, pages 13--37. Springer, Berlin, 2006.

\bibitem[F{\o}l54a]{Foln54a}
Erling F{\o}lner.
\newblock Generalization of a theorem of {B}ogolio\`uboff to topological
  abelian groups. {W}ith an appendix on {B}anach mean values in non-abelian
  groups.
\newblock {\em Math. Scand.}, 2:5--18, 1954.

\bibitem[F{\o}l54b]{Foln54b}
Erling F{\o}lner.
\newblock Note on a generalization of a theorem of {B}ogolio\`uboff.
\newblock {\em Math. Scand.}, 2:224--226, 1954.

\bibitem[HS98]{HiSt98}
Neil Hindman and Dona Strauss.
\newblock {\em Algebra in the {S}tone-\v {C}ech compactification}, volume~27 of
  {\em de Gruyter Expositions in Mathematics}.
\newblock Walter de Gruyter \& Co., Berlin, 1998.
\newblock Theory and applications.

\bibitem[Jin02]{Jin02}
Renling Jin.
\newblock The sumset phenomenon.
\newblock {\em Proc. Amer. Math. Soc.}, 130(3):855--861, 2002.

\bibitem[Jin08]{Jin08}
Renling Jin.
\newblock {\em private communication}, 2008.

\bibitem[JK03]{JiKe03}
Renling Jin and H.~Jerome Keisler.
\newblock Abelian groups with layered tiles and the sumset phenomenon.
\newblock {\em Trans. Amer. Math. Soc.}, 355(1):79--97, 2003.

\bibitem[Pat88]{Pate88}
Alan L.~T. Paterson.
\newblock {\em Amenability}, volume~29 of {\em Mathematical Surveys and
  Monographs}.
\newblock American Mathematical Society, Providence, RI, 1988.

\end{thebibliography}
\end{document}